\documentclass[11pt]{article}
\usepackage{amsfonts, amsmath, amssymb, latexsym, eucal, amscd,cite} 
\usepackage[all]{xy}
\newtheorem{theorem}{Theorem}[section]
\newtheorem{lemma}[theorem]{Lemma}
\newtheorem{proposition}[theorem]{Proposition}
\newtheorem{corollary}[theorem]{Corollary}
\newtheorem{exAux}[theorem]{Example}

\newtheorem{Def}[theorem]{Definition}
\newenvironment{definition}{\begin{Def} \rm}{\end{Def}}
\newtheorem{Note}[theorem]{Note}
\newenvironment{note}{\begin{Note} \rm}{\end{Note}}
\newtheorem{Problem}[theorem]{Problem}
\newenvironment{problem}{\begin{Problem} \rm}{\end{Problem}}
\newtheorem{Rem}[theorem]{Remark}

\newtheorem{Not}[theorem]{Notation}

\newtheorem{Conj}[theorem]{Conjecture}

\newtheorem{Ass}[theorem]{Assumption}

\newenvironment{proof}{\medskip\noindent{\bf Proof.\ }}{\qed\medskip}
\newenvironment{proofof}[1]{\medskip\noindent{\bf Proof  of {#1}.\ 
}}{\qed\medskip}
\newcommand{\qed}{\hfill\mbox{$\Box$\qquad\qquad}}

\newcommand{\F}{\mathbb{F}}
\newcommand{\Mat}{\text{\rm Mat}}

\newcommand{\vphi}{\varphi}
\renewcommand{\th}{\theta}

\newif\ifDRAFT
\DRAFTfalse

\begin{document}

\begin{center}
\LARGE\bf
Leonard pairs having LB-TD form
\end{center}

\begin{center}
\Large
Kazumasa Nomura
\end{center}

\bigskip

{
\begin{quote}
\begin{center}
{\bf Abstract}
\end{center}
Fix an algebraically closed field $\F$ and an integer $d \geq 3$.
Let $\Mat_{d+1}(\F)$ denote the $\F$-algebra consisting of the $(d+1) \times (d+1)$
matrices that have all entries in $\F$.
We consider a pair of diagonalizable matrices $A,A^*$ in $\Mat_{d+1}(\F)$,
each acts in an irreducible tridiagonal fashion on an eigenbasis for the other one.
Such a pair is called a Leonard pair in $\Mat_{d+1}(\F)$.
For a Leonard pair $A,A^*$ there is a nonzero scalar $q$ that is used to describe the eigenvalues
of $A$ and $A^*$.
In the present paper we find all Leonard pairs $A,A^*$ in $\Mat_{d+1}(\F)$
such that $A$ is lower bidiagonal with subdiagonal entries all $1$ and $A^*$
is irreducible tridiagonal,
under the assumption that $q$ is not a root of unity.
This gives a partial solution of a problem given by Paul Terwilliger.
\end{quote}
}

\section{Introduction}

Throughout the paper $\F$ denotes an algebraically closed field.
All scalars will be taken from $\F$.
Fix an integer $d \geq 0$ and a vector space $V$ over $\F$ with dimension $d+1$.
Let $\F^{d+1}$ denote the $\F$-vector space consisting of the column vectors
of length $d+1$ and $\Mat_{d+1}(\F)$ denote the $\F$-algebra consisting of
the $(d+1) \times (d+1)$ matrices.
The algebra $\Mat_{d+1}(\F)$ acts on $\F^{d+1}$ by left multiplication.

We begin by recalling the notion of a Leonard pair.
We use the following terms.
A square matrix is said to be {\em tridiagonal} whenever each nonzero
entry lies on either the diagonal, the subdiagonal, or the superdiagonal.
A tridiagonal matrix is said to be {\em irreducible} whenever
each entry on the subdiagonal is nonzero and each entry on the superdiagonal is nonzero.

\begin{definition}  {\rm \cite[Definition 1.1]{T:Leonard} }    \label{def:LP}   \samepage
\ifDRAFT {\rm def:LP}. \fi
By a {\em Leonard pair on $V$} we mean an ordered pair of linear transformations
$A : V \to V$ and $A^*: V \to V$ that satisfy (i) and (ii) below:
\begin{itemize}
\item[\rm (i)]
There exists a basis for $V$ with respect to which the matrix representing
$A$ is irreducible tridiagonal and the matrix representing $A^*$ is diagonal.
\item[\rm (ii)]
There exists a basis for $V$ with respect to which the matrix representing
$A^*$ is irreducible tridiagonal and the matrix representing $A$ is diagonal.
\end{itemize}
By a {\em Leonard pair in $\Mat_{d+1}(\F)$} we mean an ordered pair
$A,A^*$ in $\Mat_{d+1}(\F)$ that acts on $\F^{d+1}$ as a Leonard pair. 
\end{definition}

\begin{note}    \samepage
According to a common notational convention, $A^*$ denotes the
conjugate transpose of $A$.
We are not using this convention.
In a Leonard pair $A,A^*$ the matrices $A$ and $A^*$ are arbitrary subject to
the conditions (i) and (ii) above.
\end{note}

We refer the reader to
\cite{NT:span, NT:affine,T:Leonard,T:tworelations, T:LBUB, T:survey, TV}
for background on Leonard pairs.

A square matrix is said to be {\em lower bidiagonal} whenever each nonzero entry
lies on either the diagonal or the subdiagonal.
Paul Terwilliger gave the following problem:

\begin{problem} \cite[Problem 36.14]{T:survey}   \label{prob}
Find all Leonard pairs $A,A^*$ in $\Mat_{d+1}(\F)$ that satisfy the following conditions:
(i) $A$ is lower bidiagonal with subdiagonal entries all $1$;
(ii) $A^*$ is irreducible tridiagonal.
\end{problem}

The above problem is related to ``Leonard triples'' \cite{Cur,Hu},
``adjacent Leonard pairs'' \cite{Har}, and ``$q$-tetrahedron algebras'' \cite{IRT}.
In the present paper we give a partial solution of  Problem \ref{prob}.
We use the following terms:

\begin{definition}   \label{def:matLBTD}   \samepage
\ifDRAFT {\rm def:matLBTD}. \fi
An ordered pair of matrices $A,A^*$ in $\Mat_{d+1}(\F)$ is said to be {\em LB-TD}
whenever $A$ is lower bidiagonal with subdiagonal entries all $1$ and 
$A^*$ is irreducible tridiagonal.
\end{definition}

\begin{definition}    \label{def:LBTD}   \samepage
\ifDRAFT {\rm def:LBTD}. \fi
A Leonard pair $A,A^*$ on $V$ is said to have {\em LB-TD form}
whenever there exists a basis for $V$ with respect to which
the matrices representing $A,A^*$ form an LB-TD pair in $\Mat_{d+1}(\F)$.
\end{definition}

\begin{note}    \label{note:alpha0}    \samepage
\ifDRAFT {\rm note:alpha0}. \fi
Let $A,A^*$ be a Leonard pair on $V$.
For scalars $\alpha$, $\alpha^*$,
the pair $A + \alpha I$, $A^* + \alpha^* I$ is also a Leonard pair on $V$,
which is called a {\em translation} of $A,A^*$.
Here $I$ denotes the identity.
If $A,A^*$ has LB-TD form, then any translation of $A,A^*$ has LB-TD form.
\end{note}

Below we display a family of LB-TD Leonard pairs in $\Mat_{d+1}(\F)$.
Consider the following LB-TD pair in $\Mat_{d+1}(\F)$:
\begin{align}           \label{eq:LBTD}
A &=  
 \begin{pmatrix}
  \th_0  & & & & & \text{\bf 0} \\
  1 & \th_1  \\
   & 1 & \th_2 \\
   & & \cdot & \cdot \\
   & & & \cdot & \cdot \\
  \text{\bf 0}  & & & & 1 & \th_{d}
 \end{pmatrix},
& 
A^{*} &=
 \begin{pmatrix}
  x_0 & y_1 & & & & \text{\bf 0} \\
  z_1 & x_1 & y_2  \\
   & z_2 & x_2 & \cdot  \\
   & & \cdot & \cdot & \cdot \\
   & & & \cdot & \cdot & y_{d} \\
  \text{\bf 0} & & & & z_d & x_d
 \end{pmatrix}.
\end{align}

\begin{proposition}   \label{prop:ex1}  \samepage
\ifDRAFT {\rm prop:ex1}. \fi
Fix a nonzero scalar $q$ that is not a root of unity.
Let $\alpha$, $\alpha^*$, $a$, $a'$, $b$, $b'$, $c$ be scalars
with $c \neq 0$.
Define scalars
$\{\th_i\}_{i=0}^d$, $\{x_i\}_{i=0}^d$, $\{y_i\}_{i=1}^d$, $\{z_i\}_{i=1}^d$ by
\begin{align}
  \th_i &= \alpha + a q^{2i-d} + a' q^{d-2i},     \label{eq:ex1thi}
\\
  x_i &=  \alpha^* + (b+b')q^{d-2i}                 
          + a' c q^{d-2i}(q^{d+1} + q^{-d-1} - q^{d-2i-1} - q^{d-2i+1}),    \label{eq:ex1xi}
\\
  y_i &= (q^{i}-q^{-i})(q^{d-i+1}-q^{i-d-1}) 
             (b - a' c q^{d-2i+1})(b' - a' c q^{d-2i+1}) c^{-1},                 \label{eq:ex1yi}
 \\
  z_i &= - c q^{d-2i+1}.     \label{eq:ex1zi}
\end{align}
Then the matrices $A,A^*$ form an LB-TD Leonard pair in $\Mat_{d+1}(\F)$
if and only if  the scalars $a$, $a'$, $b$, $b'$, $c$ satisfy the following inequalities:
\begin{align}
& a \not\in \{ a' q^{2d-2}, \, a' q^{2d-4}, \, \ldots, \, a' q^{2-2d} \},  \label{eq:ex1cond1}
\\
& b \not\in \{ b' q^{2d-2}, \, b' q^{2d-4}, \, \ldots, \, b' q^{2-2d} \},  \label{eq:ex1cond2}
\\
& b c^{-1},\, b' c^{-1}  \not\in
 \{ a q^{d-1}, \, a q^{d-3}, \, \ldots, a q^{1-d} \} \cup
 \{ a' q^{d-1}, \, a' q^{d-3}, \, \ldots, a' q^{1-d} \}.         \label{eq:ex1cond3}
\end{align}
\end{proposition}

To state our further results, we recall some materials concerning Leonard pairs.
Consider a Leonard pair $A,A^*$ on $V$.
We first recall some facts concerning the eigenvalues of $A,A^*$.
By \cite[Lemma 1.3]{T:Leonard} each of $A,A^*$ has mutually distinct $d+1$ eigenvalues.
Let $\{\th_i\}_{i=0}^d$ be an ordering of the eigenvalues of $A$.
For $0 \leq i \leq d$ pick an eigenvector  $v_i \in V$ of $A$ 
associated with $\th_i$.
Then the ordering $\{\th_i\}_{i=0}^d$ is said to be {\em standard} whenever
the basis $\{v_i\}_{i=0}^d$ satisfies Definition \ref{def:LP}(ii).
A standard ordering of $A^*$ is similarly defined.
For a standard ordering $\{\th_i\}_{i=0}^d$ of the eigenvalues of $A$,
the ordering $\{\th_{d-i}\}_{i=0}^d$ is also standard and no further
ordering is standard.
A similar result applies to $A^*$.
Let $\{\th_i\}_{i=0}^d$ (resp.\ $\{\th^*_i\}_{i=0}^d$) be a standard ordering
of the eigenvalues of $A$ (resp.\ $A^*$).
By \cite[Theorem 1.9]{T:Leonard} the expressions
\begin{align}                 \label{eq:indep0}
  & \frac{\th_{i-2} - \th_{i+1}}
            {\th_{i-1} - \th_{i} },
  &&  \frac{\th^*_{i-2} - \th^*_{i+1}}
            {\th^*_{i-1} - \th^*_{i} }
\end{align}
are equal and independent of $\, i$ for $2 \leq i \leq d-1$.
Next we recall the notion of a parameter array of $A,A^*$.

\begin{lemma} {\rm \cite[Theorem 3.2]{T:Leonard} }     \label{lem:split}   \samepage
\ifDRAFT {\rm lem:split}. \fi
For a Leonard pair $A,A^*$ on $V$ and 
a standard ordering $\{\th_i\}_{i=0}^d$ (resp.\ $\{\th^*_i\}_{i=0}^d$) 
of the eigenvalues of $A$ (resp.\ $A^*$),
there exists a basis $\{u_i\}_{i=0}^d$ for $V$
and there exist scalars $\{\vphi_i\}_{i=1}^d$ such that
the matrices representing $A,A^*$ with respect to $\{u_i\}_{i=0}^d$ are
\begin{align}              \label{eq:splitAAs} 
A &: \;\; 
 \begin{pmatrix}
  \th_0  & & & & & \text{\bf 0} \\
  1 & \th_1  \\
   & 1 & \th_2 \\
   & & \cdot & \cdot \\
   & & & \cdot & \cdot \\
  \text{\bf 0}  & & & & 1 & \th_{d}
 \end{pmatrix},
& 
A^*  &: \;\;
 \begin{pmatrix}
  \th^*_0 & \vphi_1 & & & & \text{\bf 0} \\
     & \th^*_1 & \vphi_2  \\
   &  & \th^*_2 & \cdot  \\
   & &    & \cdot & \cdot \\
   & & &    & \cdot & \vphi_d \\
  \text{\bf 0} & & & &  & \th^*_d
 \end{pmatrix}.
\end{align}
The sequence $\{\vphi_i\}_{i=0}^d$ is uniquely determined by the ordering
$(\{\th_i\}_{i=0}^d, \{\th^*_i\}_{i=0}^d)$.
Moreover $\vphi_i \neq 0$ for $1 \leq i \leq d$.
\end{lemma}

With reference to Lemma \ref{lem:split}
we refer to $\{\vphi_i\}_{i=1}^d$ as the {\em first split sequence} of $A,A^*$
associated with the ordering $(\{\th_i\}_{i=0}^d, \{\th^*_i\}_{i=0}^d)$.
By the {\em second split sequence} of $A,A^*$ associated with
the ordering $(\{\th_i\}_{i=0}^d, \{\th^*_i\}_{i=0}^d)$ we mean the first split sequence
of $A,A^*$ associate with the ordering
 $(\{\th_{d-i}\}_{i=0}^d, \{\th^*_i\}_{i=0}^d)$.
By a {\em parameter array of $A,A^*$} we mean the sequence
\begin{equation}   \label{eq:parray}
  (\{\th_i\}_{i=0}^d, \{\th^*_i\}_{i=0}^d, \{\vphi_i\}_{i=1}^d, \{\phi_i\}_{i=1}^d),
\end{equation}
where $\{\th_i\}_{i=0}^d$ is a standard ordering of the eigenvalues of $A$,
$\{\th^*_i\}_{i=0}^d$ is a standard ordering of the eigenvalues of $A^*$,
and $\{\vphi_i\}_{i=1}^d$ (resp.\ $\{\phi_i\}_{i=1}^d$) is the first split sequence
(resp.\ second split sequence) of $A,A^*$
associated with the ordering $(\{\th_i\}_{i=0}^d, \{\th^*_i\}_{i=0}^d$).

For the Leonard pair given in Propositions \ref{prop:ex1},
the corresponding parameter array is as follows:

\begin{proposition}   \label{prop:ex1parray}   \samepage
\ifDRAFT {\rm prop:ex1parray}. \fi
With reference to Proposition \ref{prop:ex1}, assume $A,A^*$ is an LB-TD Leonard pair
in $\Mat_{d+1}(\F)$.
Define scalars $\{\th^*_i\}_{i=0}^d$, $\{\vphi_i\}_{i=1}^d$, $\{\phi_i\}_{i=1}^d$ by
\begin{align}
 \th^*_i &=  \alpha^* + b q^{2i-d} + b' q^{d-2i},       \label{eq:ex1thsi}
\\
 \vphi_i &= (q^i-q^{-i})(q^{d-i+1} - q^{i-d-1})(b - a' c q^{d-2i+1})(b' - a c q^{2i-d-1})c^{-1},
                                                                      \label{eq:ex1vphii}
\\
 \phi_i &= (q^i-q^{-i})(q^{d-i+1} - q^{i-d-1})(b - a c q^{d-2i+1})(b' - a' c q^{2i-d-1})c^{-1}.
                                                                      \label{eq:ex1phii}
\end{align}
Then \eqref{eq:parray} is a parameter array of $A,A^*$.
\end{proposition}

For the rest of this section, we assume $d \geq 3$.
Let $A,A^*$ be a Leonard pair on $V$,
and let $\{\th_i\}_{i=0}^d$ (resp.\ $\{\th^*_i\}_{i=0}^d$) be a standard ordering
of the eigenvalues of $A$ (resp.\ $A^*$).
Let $\beta$ be one less the common value of \eqref{eq:indep0}.
We call $\beta$ the {\em fundamental parameter} of $A,A^*$.
Let $q$ be a nonzero scalar such that $\beta = q^2 + q^{-2}$.
We call $q$ a {\em quantum parameter} of $A,A^*$.
We now give a solution of Problem \ref{prob} for the case that $q$ is
not a root of unity:

\begin{theorem}   \label{thm:main}   \samepage
\ifDRAFT {\rm thm:main}. \fi
Consider sequences of scalars
$\{\th_i\}_{i=0}^d$, $\{x_i\}_{i=0}^d$, $\{y_i\}_{i=1}^d$, $\{z_i\}_{i=1}^d$
such that $y_i z_i \neq 0$ for $1 \leq i \leq d$,
and consider the matrices $A,A^*$ in \eqref{eq:LBTD}.
Assume $A,A^*$ is a Leonard pair in $\Mat_{d+1}(\F)$ with quantum parameter $q$
that is not a root of unity.
Then, after replacing $q$ with $q^{-1}$ if necessary,
there exist scalars $\alpha$, $\alpha^*$, $a$, $a'$, $b$, $b'$, $c$ with $c \neq 0$ that satisfy
\eqref{eq:ex1thi}--\eqref{eq:ex1cond3}.
\end{theorem}

Let $A,A^*$ be a Leonard pair on $V$  with parameter array \eqref{eq:parray}.
Our next result gives a necessary and sufficient condition on the parameter array
for that $A,A^*$ has LB-TD form. 
To state this result, we use the following notation.
Let $q$ be a quantum parameter of $A,A^*$, and assume $q$ is not a root of unity.
By \cite[Lemma 9.2]{T:Leonard} there exist scalars
$\alpha$, $\alpha^*$, $a$, $a'$, $b$, $b'$ such that
\begin{align}
  \th_i &= \alpha + a q^{2i-d} + a' q^{d-2i}  && (0 \leq i \leq d),  \label{eq:thi}
\\
 \th^*_i &= \alpha^* + b q^{2i-d} + b' q^{d-2i} && (0 \leq i \leq d). \label{eq:thsi}
\end{align}
By \cite[Lemma 13.1]{NT:affine} there exists a scalar $\xi$ such that
\begin{align}
\vphi_i &= (q^i-q^{-i})(q^{i-d-1}-q^{d-i+1})(\xi + a b q^{2i-d-1} + a' b' q^{d-2i+1})
                                                           && (1 \leq i \leq d),    \label{eq:vphi}
\\
\phi_i &= (q^i-q^{-i})(q^{i-d-1}-q^{d-i+1})(\xi + a' b q^{2i-d-1} + a b' q^{d-2i+1})
                                                           && (1 \leq i \leq d).    \label{eq:phi}
\end{align}

\begin{theorem}   \label{thm:general}   \samepage
\ifDRAFT {\rm thm:general}. \fi
With the above notation,
the following {\rm (i)} and {\rm (ii)} are equivalent:
\begin{itemize}
\item[\rm (i)]
$A,A^*$ has LB-TD form.
\item[\rm (ii)]
At least two of $aa'$, $bb'$, $\xi$ are nonzero.
\end{itemize}
\end{theorem}

This paper is organized as follows.
In Sections \ref{sec:rec} and \ref{sec:parray} we recall some materials concerning Leonard pairs.
In Section \ref{sec:proof1} we prove Propositions \ref{prop:ex1} and \ref{prop:ex1parray}.
In Section \ref{sec:AWrel} we recall the Askey-Wilson relations.
In Section \ref{sec:useAWrel}--\ref{sec:entriesAs} we use Askey-Wilson relations
to obtain the entries of $A^*$ in \eqref{eq:LBTD}.
In Section \ref{sec:proofgeneral} we prove Theorem \ref{thm:general}.
In Section \ref{sec:proofmain} we prove Theorem \ref{thm:main}.
Leonard pairs have been classified in \cite[Section 35]{T:survey}:
there are 7 types 
for the case that $q$ is not a root of unity.
In Section \ref{sec:types} we explain about which types of Leonard pairs
have LB-TD form.

\section{Recurrent sequences}
\label{sec:rec}

In this section we recall the notion of a recurrent sequence.
We also mention some lemmas for later use.
Assume $d \geq 3$ and consider a sequence $\{\th_i\}_{i=0}^d$
consisting of mutually distinct scalars.
We say $\{\th_i\}_{i=0}^d$ is {\em recurrent} whenever
the expression
\begin{equation}     \label{eq:rec}
  \frac{\th_{i-2} - \th_{i+1} }
         {\th_{i-1} - \th_{i} }
\end{equation}
is independent of $i$ for $2 \leq i \leq d-1$.
For a scalar $\beta$, we say $\{\th_i\}_{i=0}^d$ is
{\rm $\beta$-recurrent} whenever
\begin{align*}
 \th_{i-2} - (\beta+1) \th_{i-1} + (\beta+1) \th_i - \th_{i+1} &=0  && (2 \leq i \leq d-1).
\end{align*}
Observe that $\{\th_i\}_{i=0}^d$ is recurrent if and only if it is $\beta$-recurrent
for some $\beta$.
In this case, the value of \eqref{eq:rec} is equal to $\beta+1$.

\begin{lemma}  {\rm \cite[Lemmas 8.4, 8.5]{T:Leonard} }   \label{lem:recurrent}   \samepage
\ifDRAFT {\rm lem:recurrent}. \fi
Assume $\{\th_i\}_{i=0}^d$ is $\beta$-recurrent for some scalar $\beta$. 
Then the following hold:
\begin{itemize}
\item[\rm (i)]
There exists a scalar $\gamma$ such that
\begin{align}    \label{eq:gammarec}
 \gamma &= \th_{i-1} - \beta \th_i + \th_{i+1}   && (1 \leq i \leq d-1).
\end{align}
\item[\rm (ii)]
Let $\gamma$ be from {\rm (i)}.
Then there exists a scalar $\varrho$ such that
\begin{align}      \label{eq:rhorec}
 \varrho &= \th_{i-1}^2 - \beta \th_{i-1}\th_{i} + \th_{i}^2
                 - \gamma (\th_{i-1} + \th_i)
    &&  (1 \leq i \leq d).
\end{align}
\end{itemize}
\end{lemma}

\begin{lemma}   {\rm \cite[Lemma 8.4]{T:Leonard} }  \label{lem:betarec}  \samepage
\ifDRAFT {\rm lem:betarec}. \fi
Assume $\{\th_i\}_{i=0}^d$ satisfies \eqref{eq:gammarec} for some scalars $\beta$ and $\gamma$.
Then $\{\th_i\}_{i=0}^d$ is $\beta$-recurrent.
\end{lemma}

Assume $\{\th_i\}_{i=0}^d$ is $\beta$-recurrent,
and take a nonzero scalar $q$ such that $\beta = q^2 + q^{-2}$.
Assume $q$ is not a root of unity.
By \cite[Lemma 9.2]{T:Leonard} there exist scalars $\alpha$, $a$, $a'$ such that
\begin{align}    \label{eq:thiclosed}
 \th_i &= \alpha + a q^{2i-d} + a' q^{d-2i}   && (0 \leq i \leq d).
\end{align}

\begin{lemma}   \label{lem:gammarhoclosed}   \samepage
\ifDRAFT {\rm lem:gammarhoclosed}. \fi
Let the scalars $\gamma$, $\varrho$ be from Lemma \ref{lem:recurrent}.
Then
\begin{align}
  \gamma &= - \alpha (q - q^{-1})^2,               \label{eq:gammaclosed}
\\
 \varrho &= \alpha^2 (q-q^{-1})^2 - a a' (q^2-q^{-2})^2.  \label{eq:rhoclosed}
\end{align}
\end{lemma}

\begin{proof}
Routine verification.
\end{proof}

\begin{lemma}   \label{lem:standardorder}    \samepage
\ifDRAFT {\rm lem:standardorder}. \fi
Let the scalars  $\gamma$, $\varrho$ be from Lemma \ref{lem:recurrent},
and let
$\{\tilde{\th}_i\}_{i=0}^d$ be a reordering of $\{\th_i\}_{i=0}^d$ that satisfies both
\begin{align}
 \gamma &= \tilde{\th}_{i-1} - \beta \tilde{\th}_i + \tilde{\th}_{i+1}   
           && (1 \leq i \leq d-1),                                      \label{eq:recgamma2}
\\
 \varrho &=  \tilde{\th}_{0}^2 - \beta \tilde{\th}_{0} \tilde{\th}_{1} + \tilde{\th}_{1}^2
                 - \gamma (\tilde{\th}_{0} + \tilde{\th}_{1}).       \label{eq:recrho2}
\end{align}
Then the sequence $\{\tilde{\th}_i\}_{i=0}^d$ coincides with either
$\{\th_i\}_{i=0}^d$ or $\{\th_{d-i}\}_{i=0}^d$.
\end{lemma}

\begin{proof}
Note that the sequence $\{\tilde{\th}_i\}_{i=0}^d$ is $\beta$-recurrent by \eqref{eq:recgamma2}
and Lemma \ref{lem:betarec}.
By this and \cite[Lemma 9.2]{T:Leonard} there exist scalars 
$\tilde{\alpha}$, $\tilde{a}$, $\tilde{a}'$ such that
\begin{align}     \label{eq:tildethiclosed}
  \tilde{\th}_i &= \tilde{\alpha} + \tilde{a} q^{2i-d} + \tilde{a}' q^{d-2i}
             && (0 \leq i \leq d).
\end{align}
By \eqref{eq:gammaclosed} and \eqref{eq:recgamma2} for $i=1$
one finds $\tilde{\alpha} = \alpha$.
By this and \eqref{eq:gammaclosed}, \eqref{eq:rhoclosed}, \eqref{eq:recrho2}, 
one finds $a a' =  \tilde{a} \tilde{a}'$.
Using these comments and 
the assumption that $\{\tilde{\th}_i\}_{i=0}^d$ is a permutation of 
$\{\th_i\}_{i=0}^d$,
one routinely finds that either (i) $\tilde{a} = a$, $\tilde{a}'=a'$
or (ii) $\tilde{a}=a'$, $\tilde{a}' = a$.
The result follows.
\end{proof}

\section{Parameter arrays}
\label{sec:parray}

In this section we recall some materials concerning Leonard pairs.

We first recall the notion of an isomorphism of Leonard pairs.
Consider a vector space $V'$ over $\F$ that has dimension $d+1$.
For a Leonard pairs $A,A^*$ on $V$ and a Leonard pair $B,B^*$ on $V'$,
by an {\em isomorphism of Leonard pairs} from $A,A^*$ to $B,B^*$
we mean a linear bijection $\sigma: V \to V'$ such that
both $\sigma A = B \sigma$ and $\sigma A^* = B^* \sigma$.
We say that the Leonard pairs $A,A^*$ and $B,B^*$ are {\em isomorphic} whenever 
there exists an isomorphism of Leonard pairs from $A,A^*$ to $B,B^*$.

Next we recall some facts concerning a parameter array of a Leonard pair.

\begin{definition}   \label{def:parray}
\ifDRAFT {\rm def:parray}. \fi
By a {\em parameter array over $\F$} we mean a sequence \eqref{eq:parray}
consisting of scalars in $\F$ that satisfy {\rm (i)--(v)} below:
\begin{itemize}
\item[\rm (i)]
$\th_i \neq \th_j$, $\;\;\; \th^*_i \neq \th^*_j \;\;\;$ if $\; i \neq j$ $\quad (0 \leq i,j \leq d)$.
\item[\rm (ii)]
$\vphi_i \neq 0$, $\;\;\; \phi_i \neq 0$  $\quad (1 \leq i \leq d)$.
\item[\rm (iii)]
$\vphi_i = \phi_1 \sum_{h=0}^{i-1} \frac{\th_h - \th_{d-h} } {\th_0 - \th_d} 
       + (\th^*_i - \th^*_0)(\th_{i-1} - \th_d)$
$\qquad (1 \leq i \leq d)$.
\item[\rm (iv)]
$\phi_i = \vphi_1 \sum_{h=0}^{i-1} \frac{\th_h - \th_{d-h} } {\th_0 - \th_d} 
       + (\th^*_i - \th^*_0)(\th_{d-i+1} - \th_0)$
$\qquad (1 \leq i \leq d)$.
\item[\rm (v)]
The expressions
\begin{align}   \label{eq:indep}
 & \frac{\th_{i-2} - \th_{i+1}}{\th_{i-1}-\th_i},
 &&  \frac{\th^*_{i-2} - \th^*_{i+1}}{\th^*_{i-1}-\th^*_i}
\end{align}
are equal and independent of $i$ for $2 \leq i \leq d-1$.
\end{itemize}
\end{definition}

\begin{lemma}  {\rm \cite[Theorem 1.9]{T:Leonard} } \label{lem:characterize}
\ifDRAFT {\rm lem:characterize}. \fi
Consider sequences of scalars
$\{\th_i\}_{i=0}^d$, $\{\th^*_i\}_{i=0}^d$, $\{\vphi_i\}_{i=1}^d$, $\{\phi_i\}_{i=1}^d$.
Let $A : V \to V$ and $A^* : V \to V$ be linear transformations
that are represented as in \eqref{eq:splitAAs} with respect to some basis for $V$.
Then the following {\rm (i)} and {\rm (ii)} are equivalent:
\begin{itemize}
\item[\rm (i)]
The pair $A,A^*$ is a Leonard pair with parameter array \eqref{eq:parray}.
\item[\rm (ii)]
The sequence \eqref{eq:parray} is a parameter array over $\F$.
\end{itemize}
Suppose {\rm (i)} and {\rm (ii)} hold above.
Then $A,A^*$ is unique up to isomorphism of Leonard pairs.
\end{lemma}

\begin{lemma}   \label{lem:closed2}   \samepage
\ifDRAFT {\rm lem:closed2}. \fi
Let $A,A^*$ be a Leonard pair on $V$ with parameter array \eqref{eq:parray}.
Assume $d \geq 3$, and let $q$ be a quantum parameter of $A,A^*$.
Let $\alpha$, $\alpha^*$, $a$, $a'$, $b$, $b'$, $\xi$ be scalars that satisfy
\eqref{eq:thi}--\eqref{eq:phi}.
Assume at least two of $aa'$, $bb'$, $\xi$ are nonzero.
Then there exits a nonzero scalar $c$ such that $\xi = - a a' c - b b' c^{-1}$.
Moreover, this scalar satisfies
{\small
\begin{align}
 \vphi_i &= (q^i-q^{-i})(q^{d-i+1}-q^{i-d-1})(b - a' c q^{d-2i+1})(b' - a c q^{2i-d-1})c^{-1}
            && (1 \leq i \leq d),                        \label{eq:vphi2}
\\
 \phi_i &= (q^i-q^{-i})(q^{d-i+1}-q^{i-d-1})(b - a c q^{d-2i+1})(b' - a' c q^{2i-d-1}) c^{-1}
            && (1 \leq i \leq d).                         \label{eq:phi2}
\end{align}
}
\end{lemma}

\begin{proof}
Such a scalar $c$ exists since $\F$ is algebraically closed.
To get \eqref{eq:vphi2} and \eqref{eq:phi2}, 
set $\xi = - a a' c - b b' c^{-1}$ in \eqref{eq:vphi} and \eqref{eq:phi}.
\end{proof}

\section{Proof of Propositions \ref{prop:ex1} and \ref{prop:ex1parray}}
\label{sec:proof1}

\begin{proofof}{Propositions \ref{prop:ex1} and \ref{prop:ex1parray}}
Fix a nonzero scalar $q$ that is not a root of unity.
Let $\alpha$, $\alpha^*$, $a$, $a'$, $b$, $b'$, $c$ be scalars
with $c \neq 0$.
Define scalars
$\{\th_i\}_{i=0}^d$, $\{x_i\}_{i=0}^d$, $\{y_i\}_{i=1}^d$, $\{z_i\}_{i=1}^d$ by \eqref{eq:ex1thi}--\eqref{eq:ex1zi},
and define scalars $\{\th^*_i\}_{i=0}^d$, $\{\vphi_i\}_{i=1}^d$, $\{\phi_i\}_{i=1}^d$ by
\eqref{eq:ex1thsi}--\eqref{eq:ex1phii}.
One checks that the conditions Definition \ref{def:parray}(i), (ii) are satisfied
if and only if \eqref{eq:ex1cond1}--\eqref{eq:ex1cond3} hold.
In this case, the conditions Definition \ref{def:parray}(iii)--(v) are satisfied.
Therefore \eqref{eq:parray} is a parameter array over $\F$
if and only if \eqref{eq:ex1cond1}--\eqref{eq:ex1cond3} hold.
Let the matrices $A,A^*$ be from \eqref{eq:LBTD}.
For $0 \leq r \leq d$ define $u_r \in \F^{d+1}$
that has $i$th entry
\begin{align*}
 (u_r)_i &= (-1)^{r+i} c^{i-r} q^{(d+r-i)(d-r+i-1)/2}  
\\
 & \qquad\qquad \times   
      \prod_{h=0}^{d+r-i-1} (q^{d-h} - q^{h-d})          
      \prod_{h=0}^{r-1} (b' - a c q^{2h-d+1})
      \prod_{h=0}^{d-i-1} (b' - a' c q^{2h-d+1})
\end{align*}
for $0 \leq i \leq d$.
One routinely verifies that
\begin{align*}
  A u_r &= \th_r u_r + u_{r+1}  \qquad (0 \leq r \leq d-1),
 & A u_d &= \th_d u_d,
\\
 A^* u_r &= \th^*_r u_r + \vphi_r u_{r-1} \qquad  (1 \leq r \leq d),
 & A u_0 &= \th^*_0 u_0.
\end{align*}
Therefore the matrices representing $A,A^*$ with respect to $\{u_r\}_{r=0}^d$
are as in \eqref{eq:splitAAs}.
By these comments and  Lemma \ref{lem:characterize},
$A,A^*$ is a Leonard pair with parameter array \eqref{eq:parray}
if and only if \eqref{eq:ex1cond1}--\eqref{eq:ex1cond3} hold.
In this case, observe that $y_i z_i \neq 0$ for $1 \leq i \leq d$.
The results follow.
\end{proofof}

\section{The Askey-Wilson relations}
\label{sec:AWrel}

For the rest of the paper we assume $d \geq 3$.
In this section we recall the Askey-Wilson relations for a Leonard pair.
Consider a Leonard pair $A,A^*$ on $V$ with parameter array \eqref{eq:parray}
and fundamental parameter $\beta$.
Note that $\beta$ is well-defined by our assumption $d \geq 3$.
By Lemma \ref{lem:recurrent} there exist scalars $\gamma$, $\gamma^*$, 
$\varrho$, $\varrho^*$ such that
\begin{align}
 \gamma &= \th_{i-1} - \beta \th_{i} + \th_{i+1} 
              && (1 \leq i \leq d-1),                         \label{eq:gamma}
\\
 \gamma^* &= \th^*_{i-1} - \beta \th^*_{i} + \th^*_{i+1} 
              && (1 \leq i \leq d-1),                          \label{eq:gammas}
\\
 \varrho &= \th_{i-1}^2 - \beta \th_{i-1}\th_{i} + \th_{i}^2 
             - \gamma (\th_{i-1}+\th_{i})   && ( 1 \leq i \leq d),     \label{eq:rho}
\\
 \varrho^* &= \th_{i-1}^{*2} - \beta \th^*_{i-1}\th^*_{i} + {\th_{i}^*}^2
                - \gamma^* (\th^*_{i-1}+\th^*_{i})   && (1 \leq i \leq d).     \label{eq:rhos}
\end{align}

\begin{lemma} {\rm \cite[Theorem 1.5]{TV} }  \label{lem:AWrel}   \samepage
\ifDRAFT {\rm lem:AWrel}. \fi
There exist scalars $\omega$, $\eta$, $\eta^*$ such that both
\begin{align}
  A^2 A^* - \beta A A^* A + A^* A^2 - \gamma(AA^* + A^* A) - \varrho A^*
   &= \gamma^* A^2 + \omega A + \eta I,                       \label{eq:AW1}
\\
 {A^*}^2 A - \beta A^* A A^* + A {A^*}^2 - \gamma^* (A^*A + A A^*) - \varrho^* A
  &= \gamma {A^*}^2 + \omega A^* + \eta^* I.                \label{eq:AW2}
\end{align}
The scalars $\omega$, $\eta$, $\eta^*$ are uniquely determined by $A,A^*$.
\end{lemma}

The relations \eqref{eq:AW1} and \eqref{eq:AW2} are known as the {\em Askey-Wilson relations}.
Below we describe the scalars $\omega$, $\eta$, $\eta^*$.
Define scalars $\{a_i\}_{i=0}^d$ and $\{a^*_i\}_{i=0}^d$ by
\begin{align*}
  a_i  &= \th_i + \frac{\vphi_i}{\th^*_{i}-\th^*_{i-1}} + \frac{\vphi_{i+1}}{\th^*_{i}-\th^*_{i+1}}
   \qquad\qquad (1 \leq i \leq d-1),
\\
  a_0 &= \th_0 + \frac{\vphi_1}{\th^*_0 - \th^*_1},
  \qquad\qquad  a_d = \th_d + \frac{\vphi_d}{\th^*_d - \th^*_{d-1}},
\\
  a^*_i  &= \th^*_i + \frac{\vphi_i}{\th_{i}-\th_{i-1}} + \frac{\vphi_{i+1}}{\th_{i}-\th_{i+1}}
  \qquad\qquad (1 \leq i \leq d-1),
\\
 a^*_0 &= \th^*_0 + \frac{\vphi_1}{\th_0 - \th_1},
  \qquad\qquad a^*_d = \th^*_d + \frac{\vphi_d}{\th_d - \th_{d-1}}.
\end{align*}
For notational convenience,
define $\th_{-1}$, $\th_{d+1}$ (resp.\ $\th^*_{-1}$, $\th^*_{d+1}$)
so that \eqref{eq:gamma} (resp.\ \eqref{eq:gammas}) holds for $i=0$ and $i=d$.
Let the scalars $\omega$, $\eta$, $\eta^*$ be from Lemma \ref{lem:AWrel}.

\begin{lemma} {\rm \cite[Theorem 5.3]{TV} }  \label{lem:omega}              \samepage
\ifDRAFT {\rm lem:omega}. \fi
With the above notation,
\begin{align}
  \omega &= a^*_i (\th_i-\th_{i+1}) + a^*_{i-1}(\th_{i-1}-\th_{i-2}) - \gamma^* (\th_{i-1}+\th_i)
               && (1 \leq i \leq d),                             \label{eq:omega}
\\
  \eta &= a^*_i (\th_i-\th_{i-1})(\th_i-\th_{i+1}) - \gamma^* \th_i^2 - \omega \th_i
               && (0 \leq i \leq d),                             \label{eq:eta}
\\
 \eta^* &= a_i (\th^*_i-\th^*_{i-1})(\th^*_i -\th^*_{i+1}) - \gamma {\th^*_i}^2 - \omega \th^*_i
               && (0 \leq i \leq d).                             \label{eq:etas}
\end{align}
\end{lemma}

Let $q$ be a quantum parameter of $A,A^*$, and assume $q$ is not a root of unity.
Let $\alpha$, $\alpha^*$, $a$, $a'$, $b$, $b'$, $\xi$ be
be scalars that satisfy \eqref{eq:thi}--\eqref{eq:phi}.

\begin{lemma}   \label{lem:scalars}   \samepage
\ifDRAFT {\rm lem:scalars}. \fi
With the above notation, 
assume $\alpha=0$ and $\alpha^*=0$. Then
\begin{align}
\gamma &= 0,  \qquad\qquad \gamma^*=0,       \label{eq:cgamma}
\\
  \varrho &= - a a' (q^2-q^{-2})^2,                    \label{eq:crho}
\\
  \varrho^* &= - b b' (q^2-q^{-2})^2,                \label{eq:crhos}
\\
 \omega &= (q-q^{-1})^2 \big( (q^{d+1} + q^{-d-1}) \xi - (a+a')(b+b') \big),  \label{eq:comega}
\\
 \eta &= - (q-q^{-1})(q^2-q^{-2}) \big( (a+a')\xi - a a' (b+b')(q^{d+1}+q^{-d-1}) \big), \label{eq:ceta}
\\
 \eta^* &=  - (q-q^{-1})(q^2-q^{-2}) \big( (b+b')\xi - b b' (a+a')(q^{d+1}+q^{-d-1}) \big). \label{eq:cetas}
\end{align}
\end{lemma}

\begin{proof}
The lines \eqref{eq:cgamma}--\eqref{eq:crhos} follows from Lemma \ref{lem:gammarhoclosed}.
The lines \eqref{eq:comega}--\eqref{eq:cetas} are routinely verified. 
\end{proof}

\section{Evaluating the Askey-Wilson relations}
\label{sec:useAWrel}

Let $\{\th_i\}_{i=0}^d$, $\{x_i\}_{i=0}^d$, $\{y_i\}_{i=1}^d$, $\{z_i\}_{i=1}^d$ be scalars
such that  $y_i z_i \neq 0$ for $1 \leq i \leq d$.
Consider the matrices $A,A^*$ from \eqref{eq:LBTD}, and assume
$A,A^*$ is a Leonard pair in $\Mat_{d+1}(\F)$
with fundamental parameter $\beta$ and quantum parameter $q$
that is not a root of unity.
In this section we evaluate the Askey-Wilson relations to obtain
some relations between the entries of $A$ and $A^*$.
For each matrix in $\Mat_{d+1}(\F)$ we index the rows and columns by
$0,1,\ldots,d$.

\begin{lemma}    \label{lem:standard}            \samepage
\ifDRAFT {\rm lem:standard}. \fi
With the above notation,
$\{\th_i\}_{i=0}^d$ is a standard ordering of the eigenvalues of $A$.
\end{lemma}

\begin{proof}
Clearly $\{\th_i\}_{i=0}^d$ is an ordering of the eigenvalues of $A$.
Compute the $(i-1,i+1)$-entry of \eqref{eq:AW2} for $1 \leq i \leq d-1$
to find
\[
y_i y_{i+1} (\th_{i-1} - \beta \th_i + \th_{i+1}) = y_i y_{i+1} \gamma.
\]
So $\gamma  = \th_{i-1} - \beta \th_i + \th_{i+1}$ for $1 \leq i \leq d-1$.
Compute the $(0,1)$-entry of \eqref{eq:AW2} to find
\[
  y_0 \big( \th_0^2 - \beta \th_0 \th_1 + \th_1^2 - \gamma (\th_0 + \th_1) \big) = y_0 \varrho.
\]
So $\varrho =  \th_0^2 - \beta \th_0 \th_1 + \th_1^2 - \gamma (\th_0 + \th_1)$.
By these comments and Lemma \ref{lem:standardorder} we find the result.
\end{proof}

Let $\{\th^*_i\}_{i=0}^d$ be a standard ordering of the eigenvalues of $A^*$.
Let $\alpha$, $a$, $a'$ (resp.\ $\alpha^*$, $b$, $b'$) be scalars that satisfy
\eqref{eq:thi} (resp.\ \eqref{eq:thsi}).
We assume $\alpha=0$ and $\alpha^*=0$.
Let the scalars $\gamma$, $\gamma^*$, $\varrho$, $\varrho^*$ be from
\eqref{eq:gamma}--\eqref{eq:rhos}.
Let $\{\vphi_i\}_{i=1}^d$ (resp.\ $\{\phi_i\}_{i=1}^d$) be the first split sequence
(resp.\ second split sequence) of $A,A^*$ associated with the ordering
$(\{\th_i\}_{i=0}^d, \{\th^*_i\}_{i=0}^d)$.
Let $\xi$ be a scalar that satisfies \eqref{eq:vphi} and \eqref{eq:phi},
and let the scalars $\omega$, $\eta$, $\eta^*$ be from Lemma \ref{lem:AWrel}.
Note that the scalars $\gamma$, $\gamma^*$, $\varrho$, $\varrho^*$,
$\omega$, $\eta$, $\eta^*$ are written as in \eqref{eq:cgamma}--\eqref{eq:cetas}.

\begin{lemma}   \label{lem:zi}  \samepage
\ifDRAFT {\rm lem:zi}. \fi
With the above notation, 
after replacing $q$ with $q^{-1}$ if necessary,
\begin{align}
  z_i &= z_1  q^{2-2i}  &&  (1 \leq i \leq d).         \label{eq:zi}
\end{align}
\end{lemma}

\begin{proof}
Compute the $(i+1,i-2)$-entry of \eqref{eq:AW1} to find
\begin{align}
  z_{i-1} - \beta z_i + z_{i+1} &= 0   &&  (2 \leq i \leq d-1).        \label{eq:ziaux1}
\end{align}
By \eqref{eq:ziaux1} for $i=2$
\begin{align}
  z_3 = \beta z_2 - z_1.                   \label{eq:aux2}
\end{align}
Compute the $(3,0)$-entry of \eqref{eq:AW2} to find
\begin{align*}
  z_1 z_2 - \beta z_1 z_3 + z_2 z_3 = 0.  
\end{align*}
In this equation, eliminate $z_3$ using \eqref{eq:aux2}, and simplify the result using
$\beta=q^2+q^{-2}$ to find
\[
 (z_2 - q^2 z_1)(z_2 - q^{-2} z_1) = 0.
\]
So either $z_2 = z_1 q^2$ or $z_2 = z_1 q^{-2}$.
After replacing $q$ with $q^{-1}$ if necessary, we may assume $z_2 = z_1 q^{-2}$.
Now the result follows from this and \eqref{eq:ziaux1}.
\end{proof}

For the rest of this section, we choose $q$ that satisfies \eqref{eq:zi}.

\begin{lemma}   \label{lem:xi}  \samepage
\ifDRAFT {\rm lem:xi}. \fi
With the above notation, 
for $1 \leq i \leq d-1$
\begin{equation}      \label{eq:xi}
 q^{-3} x_{i-1} - (q+q^{-1})x_i + q^3 x_{i+1} = 0.
\end{equation}
\end{lemma}

\begin{proof}
Compute the $(i+1,i-1)$-entry of \eqref{eq:AW2} to find
\[
  z_i z_{i+1} (\th_{i-1}-\beta \th_i + \th_{i+1})   
   + x_i(z_i + z_{i+1}) + x_{i-1}(z_i-\beta z_{i+1})+ x_{i+1}(z_{i+1}-\beta z_i) = 0.
\]
In this equation, the first term is zero by \eqref{eq:gamma}.
Simplify the remaining terms using \eqref{eq:zi} to find \eqref{eq:xi}.
\end{proof}

\begin{lemma}   \label{lem:xizi}  \samepage
\ifDRAFT {\rm lem:xizi}. \fi
With the above notation,
for $1 \leq i \leq d-1$
\begin{equation}    \label{eq:xizi}
 x_{i-1} - \beta x_i + x_{i+1}  = a' q^{d-2i-1}(q^2-q^{-2})(q^3-q^{-3}) z_i.
\end{equation}
\end{lemma}

\begin{proof}
Compute $(i+1,i-1)$-entry of \eqref{eq:AW1} to find
\[
  x_{i-1} - \beta x_i + x_{i+1} 
  + \th_{i-1} (z_{i+1} - \beta z_i) + \th_i (z_i + z_{i+1}) + \th_{i+1}(z_i - \beta z_{i+1}) = 0.
\]
In this line, eliminate $z_{i+1}$ using \eqref{eq:zi}, and simplify the result using \eqref{eq:thi}
to find \eqref{eq:xizi}.
\end{proof}

For notational convenience, define $y_0 = 0$ and $y_{d+1}=0$.

\begin{lemma}   \label{lem:xiyi}   \samepage
\ifDRAFT {\rm lem:xiyi}. \fi
With the above notation,
for $1 \leq i \leq d$
\begin{equation}   \label{eq:xiyi}
 y_{i-1} - \beta y_i + y_{i+1}
  =   (\th_{i-2}-\th_{i-1})x_{i-1} + (\th_{i+1}-\th_{i}) x_i + \omega.
\end{equation}
\end{lemma}

\begin{proof}
Compute the $(i,i-1)$-entry of \eqref{eq:AW1} to find
\begin{align*}
 & z_i \big( \th_{i-1}^2 - \beta \th_{i-1}\th_i + \th_i^2 - \varrho \big)  
        + y_{i-1} - \beta y_i + y_{i+1}
\\ & \qquad\qquad\qquad
  =  (\beta \th_{i-1} - \th_{i-1} - \th_i) x_{i-1}
  +  (\beta \th_i - \th_{i-1} - \th_i) x_i  
  + \omega.
\end{align*}
Simplify this using \eqref{eq:gamma} and \eqref{eq:rho} to get \eqref{eq:xiyi}.
\end{proof}

\begin{lemma}   \label{lem:xi2}
\ifDRAFT {\rm lem:xi2}. \fi
With the above notation,
for $1 \leq i \leq d$
\begin{equation}  
 x_{i-1}^2 - \beta x_{i-1} x_i + x_i^2  - \varrho^*  
   = (q+q^{-1}) \big( q^{-1} y_{i-1} - (q+q^{-1})y_i + q y_{i+1} \big) z_i.  \label{eq:xi2}
\end{equation}
\end{lemma}

\begin{proof}
Compute the $(i,i-1)$-entry of \eqref{eq:AW2}, and simplify the result using \eqref{eq:zi}
to find that the left-hand side of \eqref{eq:xi2}
is equal to $z_i$ times
\[
 q^{-2} y_{i-1} - 2 y_i + q^2 y_{i+1}
 +  (\beta \th_{i-1} - \th_{i-1} - \th_i) x_{i-1} 
 +  (\beta \th_i - \th_{i-1}  - \th_i) x_i  + \omega. 
\]
Using \eqref{eq:gamma} one finds that the above expression is equal to
\[
q^{-2} y_{i-1} - 2 y_i + q^2 y_{i+1}
 +  (\th_{i-2}-\th_{i-1}) x_{i-1} +  (\th_{i+1} - \th_i) x_i
  + \omega.        
\]
Using Lemma \ref{lem:xiyi} one finds that $z_i$ times the above expression is equal 
to the right-hand side of \eqref{eq:xi2}. 
\end{proof}

\begin{lemma}   \label{lem:y1yd}  \samepage
\ifDRAFT {\rm lem:y1yd}. \fi
With the above notation,
\begin{align}  
 (\th_0 - \beta \th_0 + \th_1) y_1 &= 
    \big((\beta - 2) \th_0^2 + \varrho \big) x_0 + \omega \th_0 + \eta,   \label{eq:y1}
\\
 (\th_{d-1} + \th_{d} - \beta \th_d) y_d &=
    \big( (\beta - 2) \th_d^2 + \varrho \big) x_d + \omega \th_d + \eta.  \label{eq:yd}
\end{align}
\end{lemma}

\begin{proof}
Compute the $(0,0)$-entry and the $(d,d)$-entry of \eqref{eq:AW1}.
\end{proof}

\begin{lemma}   \label{lem:y1}   \samepage
\ifDRAFT {\rm lem:y1}. \fi
With the above notation,
\begin{equation}   \label{eq:y12}
 \big( (1-\beta) x_0 + x_1 + (\th_0-\th_2) z_1 \big) y_1
 = (\beta-2) \th_0 x_0^2 + \omega x_0 + \varrho^* \th_0 + \eta^*.
\end{equation}
\end{lemma}

\begin{proof}
Compute the $(0,0)$-entry of \eqref{eq:AW2} to find
\[
 (2-\beta) \th_0 x_0^2 + (1-\beta)x_0 y_1 + x_1 y_1 + (2\th_0 - \beta \th_1) y_1 z_1
  - \omega x_0 - \varrho^* \th_0 - \eta^* = 0.
\]
By \eqref{eq:gamma} $2 \th_0 - \beta \th_1 = \th_0 - \th_2$.
By these comments we find \eqref{eq:y12}.
\end{proof}

\section{Obtaining the entries of $A^*$}
\label{sec:entriesAs}

Let $\{\th_i\}_{i=0}^d$, $\{x_i\}_{i=0}^d$, $\{y_i\}_{i=1}^d$, $\{z_i\}_{i=1}^d$ be scalars
such that  $y_i z_i \neq 0$ for $1 \leq i \leq d$.
Consider the matrices $A,A^*$ from \eqref{eq:LBTD}, and assume
$A,A^*$ is a Leonard pair in $\Mat_{d+1}(\F)$.
In this section we obtain the entries of $A^*$.

By Lemma \ref{lem:standard} $\{\th_i\}_{i=0}^d$ is a standard ordering of the eigenvalues of $A$.
Let $\{\th^*_i\}_{i=0}^d$ be a standard ordering of the eigenvalues of $A^*$.
Let $\beta$ (resp.\ $q$) be the fundamental parameter (resp.\ quantum parameter) of $A,A^*$,
and assume $q$ is not a root of unity.
Let $\alpha$, $a$, $a' $ (resp.\ $\alpha^*$, $b$, $b'$) be scalars that satisfy 
\eqref{eq:thi} (resp.\ \eqref{eq:thsi}).
We assume $\alpha=0$ and $\alpha^* = 0$.
Let $\{\vphi_i\}_{i=1}^d$ (resp.\ $\{\phi_i\}_{i=1}^d$) be the first split sequence
(resp.\ second split sequence) of $A,A^*$
associated with the ordering $(\{\th_i\}_{i=0}^d, \{\th^*_i\}_{i=0}^d)$.
Let $\xi$ be a scalar that satisfies \eqref{eq:vphi} and \eqref{eq:phi}.
Let the scalars $\gamma$, $\gamma^*$, $\varrho$, $\varrho^*$ be from
\eqref{eq:gamma}--\eqref{eq:rhos},
and the scalars $\omega$, $\eta$, $\eta^*$ be from Lemma \ref{lem:AWrel}.
Note that these scalars are written as in Lemma \ref{lem:scalars}.

\begin{lemma}   \label{lem:solxi}  \samepage
\ifDRAFT {\rm lem:solxi}. \fi
With the above notation,
\begin{align}                                 \label{eq:solxi}
 x_i &= q^{-2i} x_0 - a' z_1 q^{d-3i+1}(q+q^{-1})(q^i - q^{-i})  && (0 \leq i \leq d).
\end{align}
\end{lemma}

\begin{proof}
Routinely obtained from \eqref{eq:xizi} for $i=1$ and \eqref{eq:xi} for $i=1,2,\ldots,d-1$.
\end{proof}

\begin{lemma}   \label{lem:solyi}   \samepage
\ifDRAFT {\rm lem:solyi}. \fi
With the above notation,
for $1 \leq i \leq d$, $y_i$ is equal to $q^{-i}(q^i - q^{-i})$ times
\begin{align*}
 &  {a'}^2 q^{2d-3i+3}(q^{i-1}-q^{1-i})z_1 + a a' q (q+q^{-1})z_1 + b b' q^{i-2}(q^i+q^{-i})z_1^{-1}  
\\ & \qquad\qquad
 - q^{1-i}(a q^{i-d-1} + a' q^{d-i+1}) x_0 - (q^{d+1}+q^{-d-1})\xi + (a+a')(b+b').
\end{align*}
\end{lemma}

\begin{proof}
Obtained from \eqref{eq:xi2} for $i=1$ and \eqref{eq:xiyi} for $i=1,2,\ldots,d-1$
using Lemma \ref{lem:solxi}.
\end{proof}

We first consider the case that $a \neq a' q^{2d+2}$ and $a' \neq a q^{2d+2}$.

\begin{lemma}          \label{lem:solxi2}   \samepage
\ifDRAFT {\rm lem:solxi2}. \fi
With the above notation,
assume $a \neq a' q^{2d+2}$.
Then
\begin{equation}           \label{eq:solxi2}
 \xi = q^d (a a' q z_1 + b b' q^{-1} z_1^{-1}) - a q x_0 + a(b+b')q^{d+1}.
\end{equation}
\end{lemma}

\begin{proof}
In \eqref{eq:y1}, eliminate $y_1$ using Lemma \ref{lem:solyi}
and simplify the result to find
\begin{align*}
 & q^{-d-1}(q-q^{-1})(q^2-q^{-2})(a q^{-d-1} - a' q^{d+1}) 
\\
& \qquad\qquad \times
 \big(\xi - q^d (a a' q z_1 + b b' q^{-1} z_1^{-1}) + a q x_0 - a(b+b')q^{d+1} \big) = 0.
\end{align*}
By this and  $a q^{-d-1} \neq a' q^{d+1}$ we find \eqref{eq:solxi2}.
\end{proof}

\begin{lemma}   \label{lem:x0xi}   \samepage
\ifDRAFT {\rm lem:x0xi}. \fi
With the above notation,
assume $a \neq a' q^{2d+2}$ and $a' \neq a q^{2d+2}$.
Then
\begin{align}
 x_0 &= a' (q^d-q^{-d}) z_1 + (b+b') q^d,                 \label{eq:x0}
\\
  \xi &= a a' q^{1-d} z_1 + b b' q^{d-1} z_1^{-1}.     \label{eq:solxi3}
\end{align}
\end{lemma}

\begin{proof}
In \eqref{eq:yd}, eliminate $x_d$ and $y_d$ using Lemmas \ref{lem:solxi} and \ref{lem:solyi}.
Then eliminate $\xi$ using \eqref{eq:solxi2}. 
Simplify the result using \eqref{eq:thi} and Lemma \ref{lem:scalars} to find
\begin{align*}
 & q^{-d}(q-q^{-1})(q^{d+1}-q^{-d-1})(a q^{d-1} - a' q^{1-d})(a q^{d+1} - a' q^{-d-1}) 
\\  & \qquad\qquad\qquad\qquad\qquad\qquad\qquad \times
 \big( x_0 - a' (q^d-q^{-d}) z_1 - (b+b') q^d \big) = 0.
\end{align*}
We have $a' \neq a q^{2d-2}$; otherwise $\th_{d-1} = \th_d$.
By our assumption  $a' \neq a q^{2d+2}$.
By these comments we get \eqref{eq:x0}.
Line \eqref{eq:solxi3} follows from \eqref{eq:solxi2} and \eqref{eq:x0}.
\end{proof}

\begin{lemma}     \label{lem:case1}   \samepage
\ifDRAFT {\rm lem:case1}. \fi
With the above notation,
assume $a \neq a' q^{2d+2}$ and $a' \neq a q^{2d+2}$.
Then
\begin{align*}
  x_i &= (b+b') q^{d-2i} - a'  q^{1-2i} (q^{d+1}+q^{-d-1}-q^{d-2i-1}-q^{d-2i+1}) z_1
            &&   (0 \leq i \leq d),
\\
 y_i &= q^{d-1}(q^i-q^{-i})(q^{i-d-1}-q^{d-i+1})
        (b+a' z_1 q^{2-2i})(b' + a' z_1 q^{2-2i})z_1^{-1}  && (1 \leq i \leq d),
\\
 z_i &=  q^{2-2i} z_1    && (1 \leq i \leq d). 
\end{align*}
\end{lemma}

\begin{proof}
The values of $x_i$ and $y_i$ are obtained from Lemmas \ref{lem:solxi}, \ref{lem:solyi} and \eqref{eq:x0}.
The value of $z_i$ is given in Lemma \ref{lem:zi}.
\end{proof}

Next consider the case $a = a' q^{2d+2}$.
Note that $a \neq 0$ in this case; otherwise $\th_i = 0$ for $0 \leq i \leq d$.

\begin{lemma}    \label{lem:case2x0}   \samepage
\ifDRAFT {\rm lem:case2x0}. \fi
With the above notation,
assume $a = a' q^{2d+2}$.
Then
\[
 x_0 = a q^{-d-2}(1+q^2-q^{-2d}) z_1 + a^{-1} b b' q^{3d} z_1^{-1}
          - a^{-1} q^{2d+1} \xi + (b+b') q^d.
\]
\end{lemma}
 
\begin{proof}
In \eqref{eq:xiyi} for $i=d$,
eliminate $y_{d-1}$, $y_d$, $x_{d-1}$, $x_d$ using Lemmas \ref{lem:solxi} and \ref{lem:solyi}.
Then simplify the result to find that
$q^{3d-2}(q^{2d+2} - q^{-2d-2})$ times
\[
 - a x_0 + a^2 q^{-d-2}(1+q^2 - q^{-2d})z_1 +  b b' q^{3d} z_1^{-1} - q^{2d+1} \xi + a(b+b')q^d
\]
is zero. The result follows.
\end{proof}

\begin{lemma}    \label{lem:case2subcases}    \samepage
\ifDRAFT {\rm lem:case2subcases}. \fi
With the above notation,
assume $a = a' q^{2d+2}$.
Then at least one of the following {\rm (i)--(iii)} holds:
\begin{itemize}
\item[\rm (i)]
$z_1 = - a^{-1} b q^{2d}$.
\item[\rm (ii)]
$z_1 = - a^{-1} b' q^{2d}$.
\item[\rm (iii)]
$\xi = a^2 q^{-3d-1} z_1 + b b' q^{d-1} z_1^{-1}$.
\end{itemize}
\end{lemma}

\begin{proof}
In \eqref{eq:y12},
eliminate $x_1$ and $y_1$ using Lemmas \ref{lem:solxi} and \ref{lem:solyi}.
Then eliminate  $x_0$ using Lemma \ref{lem:case2x0} to find that
$z_1^{-1}q^{d-1} (q-q^{-1})(q^2 - q^{-2})$ times
\[
    (b+a z_1 q^{-2d})(b' + a z_1 q^{-2d})
        (a q^{-2d} z_1 + a^{-1} b b' q^{2d} z_1^{-1} - a^{-1} q^{d+1} \xi)
\]
is zero.
The result follows.
\end{proof}

\begin{lemma}   \label{lem:case2subcase1}   \samepage
\ifDRAFT {\rm lem:case2subcase1}. \fi
With reference to Lemma \ref{lem:case2subcases}, assume {\rm (i)} holds.
Then
\begin{align*}
 x_i &= - a^{-1} q^{2d-2i+1} \xi + b q^{-2i-1} (q^{d+1}+q^{-d-1} - q^{d-2i-1} - q^{d-2i+1})
                                    && (0 \leq i \leq d),
\\
 y_i &= q^{-d-2i-1} (q^i-q^{-i})(q^{d-i+1}-q^{i-d-1}) 
                ( q^{d+1} \xi + a b q^{-2i} + a b' q^{2i} )        && (1 \leq i \leq d),
\\
 z_i &= - a^{-1} b q^{2d-2i+2}          && (1 \leq i \leq d).
\end{align*}
\end{lemma}

\begin{proof}
The values of $x_i$ and $y_i$ are obtained from Lemmas \ref{lem:solxi}, \ref{lem:solyi}, \ref{lem:case2x0}.
The value of $z_i$ is given in Lemma \ref{lem:zi}.
\end{proof}

\begin{lemma}   \label{lem:case2subcase2}   \samepage
\ifDRAFT {\rm lem:case2subcase2}. \fi
With reference to Lemma \ref{lem:case2subcases}, assume {\rm (ii)} holds.
Then
\begin{align*}
 x_i &= - a^{-1} q^{2d-2i+1} \xi + b' q^{-2i-1} (q^{d+1}+q^{-d-1} - q^{d-2i-1} - q^{d-2i+1})
                                    && (0 \leq i \leq d),
\\
 y_i &= q^{-d-2i-1} (q^i-q^{-i})(q^{d-i+1}-q^{i-d-1}) 
                ( q^{d+1} \xi + a b q^{2i} + a b' q^{-2i} )        && (1 \leq i \leq d),
\\
 z_i &= - a^{-1} b' q^{2d-2i+2}          && (1 \leq i \leq d).
\end{align*}
\end{lemma}

\begin{proof}
Similar to the proof of Lemma \ref{lem:case2subcase1}.
\end{proof}

\begin{lemma}   \label{lem:case2subcase3}   \samepage
\ifDRAFT {\rm lem:case2subcase3}. \fi
With reference to Lemma \ref{lem:case2subcases}, assume {\rm (iii)} holds.
Then
\begin{align*}
 x_i &= (b+b')q^{d-2i} - a q^{-2d-2i-1} (q^{d+1}+q^{-d-1} - q^{d-2i-1} - q^{d-2i+1}) z_1
                                    && (0 \leq i \leq d),
\\
 y_i &=q^{d-1} (q^i-q^{-i})(q^{i-d-1}-q^{d-i+1}) 
                (b + a q^{-2d-2i} z_1)(b' + a q^{-2d-2i} z_1) z_1^{-1}        && (1 \leq i \leq d),
\\
 z_i &= q^{2-2i} z_1          && (1 \leq i \leq d).
\end{align*}
\end{lemma}

\begin{proof}
Similar to the proof of Lemma \ref{lem:case2subcase1}.
\end{proof}

Next consider the case $a' = a q^{2d+2}$.
The following lemmas can be shown in a similar way as Lemmas \ref{lem:case2x0}--\ref{lem:case2subcase3}.

\begin{lemma}    \label{lem:case3x0}   \samepage
\ifDRAFT {\rm lem:case2x0}. \fi
With the above notation,
assume $a' = a q^{2d+2}$.
Then
\[
 x_0 = a q^{3d+2} z_1 + a^{-1} b b' q^{d-2} z_1^{-1}-a^{-1}q^{-1}\xi + (b+b') q^d.
\]
\end{lemma}
 
\begin{lemma}    \label{lem:case3subcases}    \samepage
\ifDRAFT {\rm lem:case3subcases}. \fi
With the above notation,
assume $a' = a q^{2d+2}$.
Then at least one of the following {\rm (i)--(iii)} holds:
\begin{itemize}
\item[\rm (i)]
$z_1 = - a^{-1} b q^{-2}$.
\item[\rm (ii)]
$z_1 = - a^{-1} b' q^{-2}$.
\item[\rm (iii)]
$\xi = a^2 q^{d+3} z_1 + b b' q^{d-1} z_1^{-1}$.
\end{itemize}
\end{lemma}

\begin{lemma}   \label{lem:case3subcase1}   \samepage
\ifDRAFT {\rm lem:case3subcase1}. \fi
With reference to Lemma \ref{lem:case3subcases}, assume {\rm (i)} holds.
Then
\begin{align*}
 x_i &= - a^{-1} q^{-2i-1} \xi + b q^{2d-2i+1} (q^{d+1}+q^{-d-1} - q^{d-2i-1} - q^{d-2i+1})
                                    && (0 \leq i \leq d),
\\
 y_i &= q^{3d-2i+3} (q^i-q^{-i})(q^{d-i+1}-q^{i-d-1}) 
                ( q^{-d-1} \xi + a b q^{2d-2i+2} + a b' q^{2i-2d-2} )        && (1 \leq i \leq d),
\\
 z_i &= - a^{-1} b q^{-2i}          && (1 \leq i \leq d).
\end{align*}
\end{lemma}

\begin{lemma}   \label{lem:case3subcase2}   \samepage
\ifDRAFT {\rm lem:case3subcase2}. \fi
With reference to Lemma \ref{lem:case3subcases}, assume {\rm (ii)} holds.
Then
\begin{align*}
 x_i &= - a^{-1} q^{-2i-1} \xi + b' q^{2d-2i+1} (q^{d+1}+q^{-d-1} - q^{d-2i-1} - q^{d-2i+1})
                                    && (0 \leq i \leq d),
\\
 y_i &= q^{3d-2i+3} (q^i-q^{-i})(q^{d-i+1}-q^{i-d-1}) 
                ( q^{-d-1} \xi + a b q^{2i-2d-2} + a b' q^{2d-2i+2} )        && (1 \leq i \leq d),
\\
 z_i &= - a^{-1} b' q^{-2i}          && (1 \leq i \leq d).
\end{align*}
\end{lemma}

\begin{lemma}   \label{lem:case3subcase3}   \samepage
\ifDRAFT {\rm lem:case3subcase3}. \fi
With reference to Lemma \ref{lem:case3subcases}, assume {\rm (iii)} holds.
Then
\begin{align*}
 x_i &= (b+b')q^{d-2i} - a q^{2d-2i+3} (q^{d+1}+q^{-d-1} - q^{d-2i-1} - q^{d-2i+1}) z_1
                                    && (0 \leq i \leq d),
\\
 y_i &=q^{d-1} (q^i-q^{-i})(q^{i-d-1}-q^{d-i+1}) 
                (b + a q^{2d-2i+4} z_1)(b' + a q^{2d-2i+4} z_1) z_1^{-1}        && (1 \leq i \leq d),
\\
 z_i &= q^{2-2i} z_1          && (1 \leq i \leq d).
\end{align*}
\end{lemma}

\section{Proof of Theorem \ref{thm:general}}
\label{sec:proofgeneral}

Let $A,A^*$ be a Leonard pair on $V$ with parameter array \eqref{eq:parray}.
Let $q$ be a quantum parameter of $A,A^*$,
and assume $q$ is not a root of unity.
Let $\alpha$, $\alpha^*$, $a$, $a'$, $b$, $b'$, $\xi$ be scalars that satisfy
\eqref{eq:thi}--\eqref{eq:phi}.

\begin{proposition}   \label{prop:ex1exists}   \samepage
\ifDRAFT {\rm prop:ex1exists}. \fi
With the above notation, 
assume at least two of $a a'$, $b b'$, $\xi$ are nonzero.
Then $A,A^*$ is isomorphic to 
the Leonard pair given in Proposition \ref{prop:ex1}
for some nonzero scalar $c$.
\end{proposition}

\begin{proof}
Without loss of generality,
we assume $\alpha=0$ and $\alpha^*=0$.
By Lemma \ref{lem:closed2} there exists a nonzero scalar $c$
that satisfies \eqref{eq:vphi2} and \eqref{eq:phi2}.
Using conditions (i) and (ii) in Definition \ref{def:parray},
one checks that the scalars $a$, $a'$, $b$, $b'$, $c$ satisfy
the inequalities \eqref{eq:ex1cond1}--\eqref{eq:ex1cond3}.
By Proposition \ref{prop:ex1parray} the Leonard pair from Proposition \ref{prop:ex1}
has the same parameter array as $A,A^*$. 
By this and Lemma \ref{lem:characterize}
$A,A^*$ is isomorphic to the Leonard pair given in Proposition \ref{prop:ex1}.
\end{proof}

\begin{lemma}   \label{lem:onenonzero}  \samepage
\ifDRAFT {\rm lem:onenonzero}. \fi
With the above notation,
assume $A,A^*$ has LB-TD form.
Then at least one of $a a'$, $b b'$ is nonzero.
\end{lemma}

\begin{proof}
Without loss of generality,
we may assume $A,A^*$ is an LB-TD pair in $\Mat_{d+1}(\F)$.
In view of Lemma \ref{lem:standard},
we may assume that $A$ is the matrix from \eqref{eq:LBTD} after replacing
$\{\th_i\}_{i=0}^d$ with $\{\th_{d-i}\}_{i=0}^d$ if necessary.
Write $A^*$ as in \eqref{eq:LBTD}, and note that $y_i z_i \neq 0$ for $1 \leq i \leq d$.
In view of Note \ref{note:alpha0} we may assume $\alpha=0$ and $\alpha^* = 0$.
We show that at least one of $a a'$, $b b'$ is nonzero.
By way of contradiction, we assume $aa'=0$ and $bb' = 0$.
Note that  $a \neq a' q^{2d+2}$ and $a' \neq a q^{2d+2}$; otherwise both $a=0$ and $a'=0$ by $aa'=0$.
Let $\{\vphi_i\}_{i=1}^d$ (resp.\ $\{\phi_i\}_{i=1}^d$) be the first split sequence
(resp.\ second split sequence) of $A,A^*$ associated with the ordering
$(\{\th_i\}_{i=0}^d, \{\th^*_i\}_{i=0}^d)$.
Let $\xi$ be a scalar that satisfies \eqref{eq:vphi} and \eqref{eq:phi}.
By \eqref{eq:solxi3} $\xi = 0$.
By this and \eqref{eq:vphi}, \eqref{eq:phi}
\begin{align*}
 \vphi_1 &= (q-q^{-1})(q^{-d}-q^d)(a b q^{1-d} + a' b' q^{d-1}), 
\\
 \phi_1 &=  (q-q^{-1})(q^{-d}-q^d)(a' b q^{1-d} + a b' q^{d-1}). 
\end{align*}
Note that  $\vphi_1 \neq 0$ and $\phi_1 \neq 0$ by Lemma \ref{lem:characterize}(ii).
So 
\begin{align*} 
 a b q^{1-d} + a' b' q^{d-1} & \neq 0,
&
 a' b q^{1-d} + a b' q^{d-1} & \neq 0.
\end{align*}
Therefore, if $a=0$ then both $b' \neq 0$ and $b \neq 0$,
and if $a'=0$ then both $b \neq 0$ and $b' \neq 0$.
This contradicts $bb'=0$. The result follows.
\end{proof}

\begin{lemma}   \label{lem:det}   \samepage
\ifDRAFT {\rm lem:det}. \fi
With the above notation, assume $A,A^*$ has LB-TD form.
Assume either $a=a' q^{2d+2}$ or $a' = a q^{2d+2}$.
Then at least one of $bb'$, $\xi$ is nonzero.
\end{lemma}

\begin{proof}
Without loss of generality,
we may assume $A,A^*$ is an LB-TD pair in $\Mat_{d+1}(\F)$.
In view of Lemma \ref{lem:standard},
we may assume that $A$ is the matrix from \eqref{eq:LBTD} after replacing
$\{\th_i\}_{i=0}^d$ with $\{\th_{d-i}\}_{i=0}^d$ if necessary.
Write $A^*$ as in \eqref{eq:LBTD}, and note that $y_i z_i \neq 0$ for $1 \leq i \leq d$.
In view of Note \ref{note:alpha0} we may assume $\alpha=0$ and $\alpha^* = 0$.
First consider the case $a = a' q^{2d+2}$. 
Note that $a a' \neq 0$; otherwise $\th_0 = \th_1$.
By Lemma \ref{lem:case2subcases} at least one of {\rm (i)--(iii)} holds.
First assume (iii) holds.
Then at least one of $bb'$, $\xi$ is nonzero; otherwise $a z_1 =0$.
Next assume (i) holds.
By way of contradiction, assume both $b b' = 0$ and $\xi=0$.
We must have $b' =0$ since $b \neq 0$ by  $0 \neq z_1 = -a^{-1} b q^{2d}$.
By these comments and Lemma \ref{lem:case2subcase1},
\begin{align*}
 x_i &= b q^{-2i-1} (q^{d+1} + q^{-d-1} - q^{d-2i-1} - q^{d-2i+1})  && (0 \leq i \leq d),
\\
 y_i &= a b q^{-d-4i-1} (q^i - q^{-i})(q^{d-i+1} - q^{i-d-1})   && (1 \leq i \leq d),
\\
 z_i &= - a^{-1} b q^{2d-2i+2}  && (1 \leq i \leq d).
\end{align*}
We claim that $\det A^* = 0$.
To see the claim, for $0 \leq k \leq d$ we define $(k+1) \times (k+1)$ matrix $M_k$
that consists of the rows $0,1,\ldots,k$ and columns $0,1,\ldots,k$ of $A^*$.
So
\begin{align*}
M_0 &= \begin{pmatrix}
              x_0
          \end{pmatrix},  &
M_1 &= \begin{pmatrix}
              x_0 & y_1 \\
              z_1 & x_1
          \end{pmatrix}, &
M_2 &= \begin{pmatrix}
             x_0 & y_1 & 0 \\
             z_1 & x_1 & y_2 \\
             0   & z_2 & x_2
           \end{pmatrix}, &
      & \ldots &
M_d = A^*.
\end{align*}
Using induction on $k=0,1,\ldots,d$ one routinely finds that
\begin{align*}
 \det M_k &= b^{k+1} q^{-(k+1)(d+k+2)} \prod_{\ell=0}^k (1 - q^{2d-2\ell})
           && (0 \leq k \leq d).
\end{align*}
Thus $\det M_d=0$ and the claim is proved.
By elementary linear algebra, $\det A^* = \th^*_0 \th^*_1 \cdots \th^*_d$.
So $\det A^* \neq 0$ since $\th^*_i = b q^{2i-d}$ for $0 \leq i \leq d$.
This contradicts the claim.
We have shown that at least one of $\xi$, $bb'$ is nonzero for the case of (i).
Next assume (ii) holds in Lemma \ref{lem:case2subcases}.
We can show the assertion in a similar way as above.
We have shown the assertion for the case of $a = a' q^{2d+2}$.
The proof is similar for the case of $a' = a q^{2d+2}$.
\end{proof}

\begin{proofof}{Theorem \ref{thm:general}}
(i)$\Rightarrow$(ii):
By Lemma \ref{lem:onenonzero} at least one of $a a'$, $b b'$ is nonzero.
First  assume $a \neq a' q^{2d+2}$ and $a' \neq a q^{2d+2}$.
By Lemma \ref{lem:x0xi} $\xi = a a' q^{1-d} z_1 + b b' q^{d-1} z_1^{-1}$.
If one of $a a'$, $b b'$ is zero, then $\xi \neq 0$.
Thus at least two of $a a'$, $b b'$, $\xi$ are nonzero.
Next assume $a = a' q^{2d+2}$ or $a' = a q^{2d+2}$.
In this case $a a' \neq 0$; otherwise $\th_0 = \th_1$.
Moreover, at least one of $b b'$, $\xi$ is nonzero  by Lemma \ref{lem:det}.
Thus at least two of $a a'$, $b b'$, $\xi$ are nonzero.

(ii)$\Rightarrow$(i):
Follows from Proposition \ref{prop:ex1exists}.
\end{proofof}

\section{Proof of Theorem \ref{thm:main}}
\label{sec:proofmain}

\begin{proofof}{Theorem \ref{thm:main}}
Consider sequences of scalars
$\{\th_i\}_{i=0}^d$, $\{x_i\}_{i=0}^d$, $\{y_i\}_{i=1}^d$, $\{z_i\}_{i=1}^d$ such that
$y_i z_i \neq 0$ for $1 \leq i \leq d$,
and consider the matrices $A,A^*$ from \eqref{eq:LBTD}.
Assume $A,A^*$ is an LB-TD Leonard pair in $\Mat_{d+1}(\F)$ with quantum parameter $q$
that is not a root of unity.
By Lemma \ref{lem:standard} $\{\th_i\}_{i=0}^d$ is a standard ordering of the eigenvalues of $A$.
Let $\{\th^*_i\}_{i=0}^d$ be a standard ordering of the eigenvalues of $A^*$.
Let $\beta$ (resp.\ $q$) be the fundamental parameter (resp.\ quantum parameter) of $A,A^*$,
and assume $q$ is not a root of unity.
Let $\alpha$, $a$, $a' $ (resp.\ $\alpha^*$, $b$, $b'$) be scalars that satisfy 
\eqref{eq:thi} (resp.\ \eqref{eq:thsi}).
We may assume $\alpha=0$ and $\alpha^* = 0$.
Let $\{\vphi_i\}_{i=1}^d$ (resp.\ $\{\phi_i\}_{i=1}^d$) be the first split sequence
(resp.\ second split sequence) of $A,A^*$
associated with the ordering $(\{\th_i\}_{i=0}^d, \{\th^*_i\}_{i=0}^d)$.
Let $\xi$ be a scalar that satisfies \eqref{eq:vphi} and \eqref{eq:phi}.
Note that at least two of $a a'$, $b b'$, $\xi$ are nonzero
by Theorem \ref{thm:general}, and so Lemma \ref{lem:closed2} applies.

First consider the case that $a \neq a' q^{2d+2}$ and $a' \neq a q^{2d+2}$.
By Lemma \ref{lem:case1} the scalars $\{x_i\}_{i=0}^d$, $\{y_i\}_{i=1}^d$, $\{z_i\}_{i=1}^d$
are written as in Lemma \ref{lem:case1}.
Setting  $z_1 = - c q^{d-1}$ in these expressions,
we obtain \eqref{eq:ex1xi}--\eqref{eq:ex1zi}.

Next consider the case $a = a' q^{2d+2}$.
By Lemma \ref{lem:case2subcases} at least one of (i)--(iii) holds in that lemma.
First assume (i) holds.
By Lemma \ref{lem:case2subcase1} the scalars  $\{x_i\}_{i=0}^d$, $\{y_i\}_{i=1}^d$, $\{z_i\}_{i=1}^d$
are written as in Lemma \ref{lem:case2subcase1}.
Setting $\xi = - a a' c - b b' c^{-1}$ in these expressions,
and replacing $(b,b',c)$ with
$(a c q^{-d-1},\, a^{-1} b b' c^{-1} q^{d+1}, \, a^{-1}b q^{d+1})$,
we obtain \eqref{eq:ex1xi}--\eqref{eq:ex1zi}.
Next assume (ii) holds.
Observe the expressions in Lemma \ref{lem:case2subcase2} are obtained
from the expressions in Lemma \ref{lem:case2subcase1}
by exchanging $b$ and $b'$.
Now proceed as above after exchanging $b$ and $b'$.
Next assume (iii) holds.
By Lemma \ref{lem:case2subcase3}  the scalars  $\{x_i\}_{i=0}^d$, $\{y_i\}_{i=1}^d$, $\{z_i\}_{i=1}^d$
are written as in Lemma \ref{lem:case2subcase3}.
Setting $z_1=- c q^{d-1}$ in these expresssions,
we obtain \eqref{eq:ex1xi}--\eqref{eq:ex1zi}.

Next consider the case $a' = a q^{2d+2}$.
By Lemma \ref{lem:case3subcases} at least one of (i)--(iii) holds in that lemma.
First assume Lemma \ref{lem:case3subcases}(i) holds.
By Lemma \ref{lem:case3subcase1} the scalars  $\{x_i\}_{i=0}^d$, $\{y_i\}_{i=1}^d$, $\{z_i\}_{i=1}^d$
are written as in Lemma \ref{lem:case3subcase1}.
Setting $\xi = - a a' c - b b' c^{-1}$ in these expressions,
and replacing $(b,b',c)$ with
$(a c q^{d+1},\, a^{-1} b b' c^{-1} q^{-d-1}, \, a^{-1} b q^{-d-1})$,
we obtain \eqref{eq:ex1xi}--\eqref{eq:ex1zi}.
Next assume Lemma \ref{lem:case3subcases}(ii) holds.
Observe the expressions in Lemma \ref{lem:case3subcase2} are obtained
from the expressions in Lemma \ref{lem:case3subcase1}.
Now proceed as above after exchanging $b$ and $b'$.
Next assume Lemma \ref{lem:case3subcases}(iii) holds.
By Lemma \ref{lem:case3subcase3}  the scalars  $\{x_i\}_{i=0}^d$, $\{y_i\}_{i=1}^d$, $\{z_i\}_{i=1}^d$
are written as in Lemma \ref{lem:case3subcase3}.
Setting $z_1=- c q^{d-1}$ in these expresssions,
we obtain \eqref{eq:ex1xi}--\eqref{eq:ex1zi}.

We have shown that there exist scalars $\alpha$, $\alpha^*$, $a$, $a'$, $b$, $b'$, $c$ with $c \neq 0$
that satisfy \eqref{eq:ex1xi}--\eqref{eq:ex1zi}.
By the construction, \eqref{eq:ex1thi} holds.
By Proposition \ref{prop:ex1} these scalars satisfy the inequalities 
\eqref{eq:ex1cond1}--\eqref{eq:ex1cond3}.
\end{proofof}

\section{Types of Leonard pairs}
\label{sec:types}

Let $A,A^*$ be a Leonard pair on $V$ with quantum parameter
$q$ that is not a root of unity.
Let 
\[
  (\{\th_i\}_{i=0}^d, \{\th^*_i\}_{i=0}^d, \{\vphi_i\}_{i=1}^d, \{\phi_i\}_{i=1}^d)
\]
be a parameter array of $A,A^*$.
By \cite[Lemma 9.2]{T:Leonard} there exist scalars $\alpha$, $\alpha^*$,
$a$, $a'$, $b$, $b'$ such that
\begin{align*}
 \th_i &= \alpha + a q^{2i-d} + a' q^{d-2i}  && (0 \leq i \leq d),
\\
 \th^*_i &= \alpha^* + b q^{2i-d} + b' q^{d-2i} && (0 \leq i \leq d).
\end{align*}
By \cite[Lemma 13.1]{NT:affine} there exists a scalar $\xi$ such that
\begin{align*}
\vphi_i &= (q^i-q^{-i})(q^{i-d-1}-q^{d-i+1})(\xi + a b q^{2i-d-1} + a' b' q^{d-2i+1})
                                                           && (1 \leq i \leq d), 
\\
\phi_i &= (q^i-q^{-i})(q^{i-d-1}-q^{d-i+1})(\xi + a' b q^{2i-d-1} + a b' q^{d-2i+1})
                                                           && (1 \leq i \leq d).
\end{align*}
By \cite[Section 35]{T:survey},
according to the values of $a$, $a'$, $b$, $b'$, $\xi$,
the Leonard pair $A,A^*$ has one of the following types:
\[
\begin{array}{rrrrr|c}
    a \;\;  & \qquad a' \;  & \qquad b \;\; & \qquad b' \; & \qquad \xi \;\;
  & \text{Name}
\\ \hline\hline \rule{0mm}{4.5mm}
   \neq 0  & \neq 0 & \neq 0 & \neq 0 & \text{any} & \text{$q$-Racah}
\\ \hline  \rule{0mm}{4.5mm}
  0 & \neq 0 & \neq 0 & \neq 0 &  \neq 0   & \text{$q$-Hahn}  
 \\
  \neq 0 & 0 & \neq 0 & \neq 0 & \neq 0 
\\ \hline  \rule{0mm}{4.5mm}
  \neq 0 & \neq 0 & 0 & \neq 0 & \neq 0   & \text{dual $q$-Hahn}
  \\
  \neq 0 & \neq 0 & \neq 0 & 0 & \neq 0
\\ \hline  \rule{0mm}{4.5mm}
  \neq 0 & 0 & 0 & \neq 0 & \neq 0   & \;\;  \text{quantum $q$-Krawtchouk}
  \\
  0 & \neq 0 & \neq 0 & 0 & \neq 0
\\ \hline \rule{0mm}{4.5mm}
  0 & \neq0 & \neq 0 & \neq 0 & 0 & \text{$q$-Krawtchouk}
  \\
  \neq 0 & 0 & \neq 0 & \neq 0 & 0
\\ \hline \rule{0mm}{4.5mm}
  0 & \neq 0 & 0 & \neq 0 & \neq 0 & \text{affine $q$-Krawtchouk}
 \\
 \neq 0 & 0 & \neq 0 & 0 & \neq 0
\\ \hline \rule{0mm}{4.5mm}
 \neq 0 & \neq 0 & 0 & \neq 0 & 0 & \text{dual $q$-Krawtchouk}
 \\
 \neq 0 & \neq 0 & \neq 0 & 0 & 0
\end{array}
\]
By Theorem \ref{thm:general} $A,A^*$ has LB-TD form if and only if at least two of
$aa'$, $bb'$, $\xi$ are nonzero.
Therefore we obtain:

\begin{corollary}  \samepage
Let $A,A^*$ be  a Leonard pair on $V$ with quantum parameter
$q$ that is not a root of unity.
Then the following {\rm (i)} and {\rm (ii)} are equivalent:
\begin{itemize}
\item[\rm (i)]
$A,A^*$ has LB-TD form.
\item[\rm (ii)]
$A,A^*$ has one of the types:
$q$-Racah, $q$-Hahn, dual $q$-Hahn.
\end{itemize}
\end{corollary}

\section{Acknowledgments}

The author thanks the referee for many insightful comments
that lead to great improvements in the paper.

\bigskip\bigskip\noindent
Kazumasa Nomura\\
Professor Emeritus\\
Tokyo Medical and Dental University\\
Kohnodai, Ichikawa, 272-0827 Japan\\
email: knomura@pop11.odn.ne.jp

\medskip\noindent
{\small
{\bf Keywords.} Leonard pair, tridiagonal pair, Askey-Wilson relation, orthogonal polynomial
\\
\noindent
{\bf 2010 Mathematics Subject Classification.} 05E35, 05E30, 33C45, 33D45
}

\end{document}